\documentclass[german,reqno]{amsart}
\usepackage{amssymb}
\usepackage[mathscr]{eucal}
\usepackage{hyperref}
\usepackage{xcolor}
\usepackage{graphicx}

\theoremstyle{plain}
\newtheorem{prop}{Proposition}[section]
\newtheorem{thm}[prop]{Theorem}
\newtheorem{cor}[prop]{Corollary}
\newtheorem{lem}[prop]{Lemma}

\theoremstyle{definition}
\newtheorem{dfn}[prop]{Definition}
\newtheorem{rem}[prop]{Remark}
\newtheorem{rems}[prop]{Remarks}
\newtheorem{example}[prop]{Example}

\newtheorem{lab}[prop]{}

\numberwithin{equation}{section}

\newcommand{\tot}{\leftrightarrow}
\newcommand{\into}{\hookrightarrow}
\newcommand{\onto}{\twoheadrightarrow}

\newcommand{\A}{{\mathbb{A}}}
\newcommand{\C}{{\mathbb{C}}}
\newcommand{\N}{{\mathbb{N}}}
\renewcommand{\P}{{\mathbb{P}}}
\newcommand{\R}{{\mathbb{R}}}
\newcommand{\Z}{{\mathbb{Z}}}

\newcommand{\scrI}{{\mathscr{I}}}

\DeclareMathOperator{\conv}{conv}

\renewcommand{\div}{{\rm div}}

\DeclareMathOperator{\Hom}{Hom}

\DeclareMathOperator{\Spec}{Spec}

\newcommand{\Th}{\mathrm{TH}}

\newcommand{\x}{{\mathtt{x}}}

\newcommand{\du}{{\scriptscriptstyle\vee}}
\renewcommand{\emptyset}{\varnothing}
\renewcommand{\setminus}{\smallsetminus}
\newcommand{\ol}{\overline}
\newcommand{\plus}{{\scriptscriptstyle+}}
\newcommand{\all}{\forall\,}
\newcommand{\ex}{\exists\,}

\newcommand{\sa}{semi-algebraic}

\newcommand\http[1]{\href{http://#1}{\nolinkurl{#1}}}

\begin{document}

\title
  [Convex hulls of curves of genus one]
  {Convex hulls of curves of genus one}

\author
  {Claus Scheiderer}
\address
  {Fachbereich Mathematik and Statistik \\
  Universit\"at Konstanz \\
  78457 Konstanz \\
  Germany}
\email
  {claus.scheiderer@uni.constanz.de}
\urladdr
  {http://www.math.uni-konstanz/\textasciitilde scheider}

\begin{abstract}
Let $C$ be a real nonsingular affine curve of genus one, embedded in
affine $n$-space, whose set of real points is compact. For any
polynomial $f$ which is nonnegative on $C(\R)$, we prove that there
exist polynomials $f_i$ with $f\equiv\sum_if_i^2$ (mod~$\scrI_C$) and
such that the degrees $\deg(f_i)$ are bounded in terms of $\deg(f)$
only. Using Lasserre's relaxation method, we deduce an explicit
representation of the convex hull of $C(\R)$ in $\R^n$ by a lifted
linear matrix inequality. This is the first instance in the
literature where such a representation is given for the convex hull
of a nonrational variety. The same works for convex hulls of
(singular) curves whose normalization is $C$. We then make a detailed
study of the associated degree bounds. These bounds are directly
related to size and dimension of the projected matrix pencils.
In particular, we prove that these bounds tend to infinity when the
curve $C$ degenerates suitably into a singular curve, and we provide
explicit lower bounds as well.
\end{abstract}

\thanks
  {}


\subjclass[2000]
  {}

\maketitle

%===================================================================%

\section*{Introduction}

Let $V\subset\A^n$ be an affine algebraic variety over $\R$ whose set
$V(\R)$ of real points is compact.
The convex hull of $V(\R)$ in $\R^n$ is a compact \sa\ set. Recently
there has been a growing interest in describing this set, or its
boundary, from different perspectives, see \cite{paba}, \cite{he},
\cite{gpt}, \cite{sss}, \cite{rs}. Part of the motivation comes from
potential applications in semidefinite programming. If $A_i$ ($i=0,
\dots,n$) are symmetric real matrices of some fixed size, an
inequality
$$A_0+x_1A_1+\cdots+x_nA_n\>\succeq\>0$$
is called a linear matrix inequality (LMI) in the variables $x_1,
\dots,x_n$. (Here $\succeq$ denotes positive semidefiniteness of the
matrix.) The set $K$ of $x\in\R^n$ which satisfy the LMI is a basic
closed and convex \sa\ subset of $\R^n$. From the view point of
convex optimization, such a description is very useful since it
allows quick and efficient optimization of linear functions on $K$,
see e.g.\ \cite{nn}, \cite{bv}, \cite{ne}.

Convex sets which allow an LMI representation are also called
spectrahedra. Being a spectrahedron is a restrictive property for
basic closed convex sets, since these sets are rigidly convex, a
property which is much stronger than just convexity \cite{hv}. In
dimension $\le2$, rigid convexity characterizes spectrahedra
(\cite{hv} Thm.\ 2.2). In higher dimensions it is currently unknown
whether such a converse holds.

For optimization purposes, however, a linear projection of a
spectrahedron works just as well as a spectrahedron itself. That
$K\subset\R^n$ is a projected spectrahedron means that there exist
symmetric real matrices $A_i$ ($0\le i\le n$) and $B_j$ ($1\le j\le
m$) such that $K$ is the set of $x\in\R^n$ for which there exists
$y\in\R^m$ with
$$A_0+\sum_{i=1}^nx_iA_i+\sum_{j=1}^my_jB_j\>\succeq\>0.$$
One speaks of a lifted LMI representation of $K$, or of a
semidefinite (SDP) representation. Projected spectrahedra form a much
wider class than spectrahedra, and much research effort is currently
spent on understanding their properties, e.g.\ \cite{nn}, \cite{ne},
\cite{la}, \cite{hn2}, \cite{hn1}, \cite{nps}, \cite{nt}, \cite{gn}.
In fact, Helton and Nie \cite{hn2} have conjectured that every convex
\sa\ set allows a lifted LMI representation.

Obtaining explicit lifted LMI representations for concretely given
convex sets is a different matter. A general construction, called the
relaxation method, is due to Lasserre \cite{la} and applies in many
cases. We will recall it (in specialized form) in Section~1 below.
Other constructions are due to Helton and Nie (\cite{hn2},
\cite{hn1}), who proved the existence of lifted LMI representations
for several large classes of convex sets.

Here we are interested in applying Lasserre's construction to the
convex hull of a (compact) real algebraic curve $C$ in the affine
plane or some higher-dimensional space. The key properties that are
needed to make the relaxation method work are a partial stability
property and a partial saturation property, each for the cone of sums
of squares in the coordinate ring $\R[C]$ (see Sect.~1). Namely,
every linear polynomial that is nonnegative on the curve has to be a
sum of squares in $\R[C]$ with uniformly bounded degrees.

Our results apply when the curve $C$ is nonsingular of genus one and
its real part $C(\R)$ is compact. It has been known for some time
already that every psd element in $\R[C]$ is a sum of squares. We
prove that the sums of squares cone in $\R[C]$ is stable, which is
our main result (Theorems \ref{sosstable}, \ref{thmsum}). The proof
uses algebraic-geometric methods, and unfortunately it seems to be
restricted to genus one. No similar result is known for any curve of
genus $>1$ (with compact real points). On the other hand, our result
gives the first construction of a lifted LMI representation for the
convex hull of a nonrational real algebraic variety. We illustrate
the application to such representations by means of some concrete
examples (Sect.~3).

Since the explicit nature of the stability (degree) bounds is
responsible for the sizes of the lifted LMI representations, there
exist good reasons to study these bounds in more detail. This is
mainly done in Sect.~4. We succeed in making the bounds fairly
explicit, and in a sense we arrive at the best possible bounds. As a
result, we can make the lifted LMI representations completely
explicit for many curves. We also study how the bounds change under
variation of the curve, and we prove that they tend to infinity when
the curve gets degenerated to a singular (rational) curve.

%-------------------------------------------------------------------%

\section{Convex hulls of algebraic sets and Lasserre relaxation}

We give a brief review here of Lasserre's relaxation method for the
construction of lifted LMI representations, however only in the
special case which will be used later, to keep the exposition less
technical.

\begin{lab}
For the following discussion, $A$ can be any finitely generated
$\R$-algebra. Let $V=\Spec(A)$ be the associated affine $\R$-variety.
The set $V(\R)=\Hom_\R(A,\R)$ of $\R$-algebra homomorphisms has a
natural euclidean topology, namely the topology induced by the
inclusion $V(\R)\into\R^n$, $p\mapsto(x_1(p),\dots,x_n(p))$, where
$x_1,\dots,x_n$ is any system of generators of $A$. This embedding
identifies $V(\R)$ with a (closed) real algebraic subset of $\R^n$.
As usual, we think of the elements $p\in V(\R)$ as points and denote
the pairing between $f\in A$ and $p\in V(\R)$ by $f(p)$.
\end{lab}

\begin{lab}\label{condsex}
Let $\Sigma A^2$ denote the cone of sums of squares in $A$. By
$$A_\plus\>=\>\{f\in A\colon\all p\in V(\R)\ f(p)\ge0\}$$
we denote the cone of all \emph{positive semidefinite (psd)} elements
of $A$. Given any finite-dimensional linear subspace $L$ of $A$, one
can ask two questions:
\begin{itemize}
\item[(1)]
Is $L\cap A_\plus$ contained in $\Sigma A^2$ (and hence equal to
$L\cap\Sigma A^2$)?
\item[(2)]
Does there exist a finite-dimensional linear subspace $W$ of $A$ such
that every $f\in L\cap\Sigma A^2$ can be written $f=\sum_{i=1}^r
a_i^2$ with $r\in\N$ and $a_1,\dots,a_r\in W$?
\end{itemize}
Recall that the preordering $\Sigma A^2$ is called \emph{saturated}
if $A_\plus=\Sigma A^2$ (\cite{sch:tams}, \cite{sch:guide}).
Therefore, a positive answer to (1) can be regarded as a partial
saturatedness property of $\Sigma A^2$. On the other hand, $\Sigma
A^2$ is called \emph{stable} if (2) has a positive answer for any
finite-dimensional $L$ (\cite{ps}, \cite{sch:stable}). Therefore, a
positive answer to (2) means a partial stability property of $\Sigma
A^2$.
\end{lab}

\begin{rem}\label{citgpt1}
Assume we are fixing a system of generators of $A$, so that $A=\R[\x]
/I$ for some ideal $I$ of $\R[\x]$, where $\x=(x_1,\dots,x_n)$ is a
tuple of variables. For $d\ge0$ let $\R[\x]_d$ be the space of
polynomials of total degree $\le d$ in $\R[\x]$, and put $A_d=(\R[\x]
_d+I)/I$. Given integers $d$, $k\ge0$, the ideal $I$ is said to be
\emph{$(d,k)$-sos} in \cite{gpt} if (1) and (2) hold for $L=A_d$ and
$W=A_k$. The problem of characterizing the $(1,k)$-sos ideals in
$\R[\x]$, and in particular the $(1,1)$-sos ideals, was raised by
Lov\'asz \cite{lo}, who showed that this question for certain
$0$-dimensional ideals is closely related to the stable set problem
for graphs.
\end{rem}

\begin{lab}\label{lasrelax}
We now recall Lasserre's important relaxation construction \cite{la}.
Assume $A=\R[\x]/I$ for some ideal $I$ of $\R[\x]$, where $\x=(x_1,
\dots,x_n)$. We denote the zero set of $I$ in $\R^n$ by $V_\R(I)$.
For convenience of exposition let us assume that $I$ does not contain
any nonzero polynomial of degree~$\le1$.

Let $L=A_1=\{f+I\colon f\in\R[\x]$, $\deg(f)\le1\}\subset A$, and let
$W$ be some finite-dimensional linear subspace of $A$ containing $L$.
Let $U$ be the linear subspace of $A$ generated by all squares $a^2$
with $a\in W$; clearly $L\subset U$ and $\dim(U)<\infty$. Let $\rho
\colon U^\du\to L^\du$ be the restriction map between the dual linear
spaces induced by the inclusion $L\subset U$. Moreover, let $U_1^\du$
(resp.\ $L_1^\du$) denote the set of all linear forms $\lambda$ in
$U^\du$ (resp.\ in $L^\du$) with $\lambda (1)=1$. Then $L_1^\du$ is
canonically identified with $\R^n$ via $\lambda\tot(\lambda(\ol x_1),
\dots,\lambda(\ol x_n))$, and we always consider $V_\R(I)$ as a real
algebraic subset of $L_1^\du=\R^n$ in the natural way.

Let $M_W=\bigl\{\sum_{i=1}^ra_i^2\colon r\in\N$, $a_1,\dots,a_r\in
W\}\subset A$ denote the set of sums of squares of elements of $W$.
This is a convex cone in $U$, which is closed in $U$ if $I$ is a real
radical ideal (\cite{ps} Prop.\ 2.6). We'll denote the dual of a
convex cone $C$ by $C^*$, so $M_W^*$ is the dual cone of $M_W$ in
$U^\du$. Then $M_W^*$ is a spectrahedron in $U^\du$, which means that
$M_W^*$ can be defined in $U^\du$ by a (homogeneous) linear matrix
inequality.
Indeed, for $\mu\in U^\du$ the symmetric bilinear form
$$\beta(\mu)\colon\ W\times W\to\R,\quad(a,a')\mapsto\mu(aa')$$
depends linearly on $\mu$, and by definition it is psd if and only if
$\mu\in M_W^*$.

The subset $M_W^*\cap U_1^\du$ of $M_W^*$ is an affine-linear section
of $M_W^*$, and is therefore a spectrahedron as well. Its image
$$K_W\>:=\>\rho(M_W^*\cap U_1^\du)\>=\>L_1^\du\cap\rho(M_W^*)$$
under the restriction map $\rho\colon U_1^\du\to L_1^\du=\R^n$ is a
convex \sa\ subset of $L_1^\du=\R^n$ which contains $V_\R(I)$. By
construction, $K_W$ is a linear projection of a spectrahedron.
Increasing $W$ results in decreasing $K_W$, so by making $W$ larger
and larger (of finite dimension) one gets a shrinking family of
convex sets $K_W$ which all contain $V_\R(I)$. For ease of exposition
let us assume that the ideal $I$ is real radical.
Then the closure $\ol{K_W}$ is equal to $L_1^\du\cap(L\cap M_W)^*$.
Moreover, $\ol{K_W}=\ol{\conv V_\R(I)}$ holds if, and only if, $L$
and $W$ satisfy conditions (1) and (2) of \ref{condsex}. (See
\cite{la} Thm.~2 and \cite{nps} Prop.\ 3.1.) If these conditions are
fulfilled, and if the convex hull of $V_\R(I)$ is closed, we have
obtained an explicit representation of $\conv V_\R(I)=K_W$ by a
lifted LMI. Note that $\conv V_\R(I)$ will be automatically closed if
the real algebraic set $V_\R(I)$ is compact.
\end{lab}

\begin{lab}
We keep the assumptions and notations of \ref{lasrelax}. Let us take
$W=A_k=(\R[\x]_d+I)/I$ for some $k\ge1$, and form the associated
projected spectrahedron $K_{A_k}$ as before. By \ref{lasrelax}, the
ideal $I$ is $(1,k)$-sos (see \ref{citgpt1}) if, and only if,
$\ol{\conv V_\R(I)}=\ol{K_{A_k}}$. The \emph{$k$-th theta body} of
the ideal $I$, defined in \cite{gpt} as
$$\Th_k(I)\>=\>\bigl\{x\in\R^n\colon\all f\in L\cap M_{A_k}\ f(x)\ge0
\bigr\},$$
is (by definition) equal to $L_1^\du\cap(L\cap M_{A_k})^*$, and is
therefore equal to $\ol{K_{A_k}}$.
The ideal $I$ is said to be \emph{$\Th_k$-exact} in \cite{gpt} if
$\Th_k(I)$ is the closure of $\conv V_\R(I)$. We see that this is the
case if and only if $I$ is $(1,k)$-sos (assuming $I$ real radical),
see \cite{gpt} Prop.\ 2.8.
\end{lab}

\begin{lab}\label{verorem}
Conditions (1) and (2) of \ref{condsex} are also of interest for
finite-dimensional subspaces $L$ of $A$ other than $L=A_1$. Given an
arbitrary such subspace $L$ (containing~$1$), let $A'$ be the
$\R$-subalgebra generated by $L$, and let $V'=\Spec(A')$. If
$1=u_0,u_1,\dots,u_m$ is a vector space basis of $L$, then $x\mapsto
(u_1(x),\dots,u_m(x))$ is a closed embedding of $V'$ into affine
$m$-space. If there exists a finite-dimensional subspace $W$ of $A$
satisfying (1) and (2), and if $V'(\R)$ is compact, we get a
representation of the convex hull of $V'(\R)$ in $\R^m$ as a
projected spectrahedron.
\end{lab}

\begin{rem}
In the above discussion we only considered sums of squares,
corresponding on the geometric side to convex hulls of real algebraic
sets in $\R^n$. We did so to simplify the exposition, and since the
main results of this paper only concern this case. Note however that
both the setup and the results of Lasserre relaxation generalize well
to arbitrary finitely generated quadratic modules. On the geometric
side, this corresponds to convex hulls of basic closed \sa\ sets. See
also \cite{nps} and \cite{gn} for more details.
\end{rem}

%-------------------------------------------------------------------%

\section{Stability of sums of squares}

From now on we will consider affine algebraic curves (mostly
nonsingular) of genus one. By the genus of a real curve which is
irreducible over $\C$, we mean the (geometric) genus of its
nonsingular projective model. In this section we will prove:

\begin{thm}\label{sosstable}
Let $C$ be an irreducible nonsingular affine curve of genus one over
$\R$ which has at least one pair of conjugate nonreal points at
infinity. Then the preordering of sums of squares in $\R[C]$ is
stable and saturated.
\end{thm}

\begin{rems}\label{remsthmsosst}
\hfil
\smallskip

1.\
See \ref{condsex} for the meaning of stable or saturated. Theorem
\ref{sosstable} says that questions (1) and (2) in \ref{condsex} have
a positive answer for any finite-dimensional linear subspace $L$ of
$\R[C]$. Therefore, if $C(\R)$ is compact, the relaxation
construction \ref{lasrelax} applies and gives lifted LMI
representations of the convex hull of $C(\R)$ for any closed
embedding of $C$ into affine space. We will discuss these
applications in more detail in Section~3 below.
\smallskip

2.\
That $\Sigma\R[C]^2$ is saturated, i.e.\ that $\rm psd=sos$ holds on
$C$, was already proved in \cite{sch:tams} (and again, in much
greater generality, in \cite{sch:mz}). A special case of the
stability part of Theorem \ref{sosstable} was mentioned in \cite{ps}
(Example 2.17) without proof. The key argument was sketched in
\cite{pl} (unpublished).
\end{rems}

\begin{lab}
We always denote the nonsingular projective completion of $C$ by
$\ol C$. The (geometric) points of $C$ at infinity are by definition
the points in $\ol C(\C)\setminus C(\C)$. The condition that $C$ has
at least one nonreal point at infinity says that at least one among
these points is not real.
\end{lab}

We first show that it suffices to prove Theorem \ref{sosstable} in
the case where $\ol C(\C)\setminus C(\C)$ consists of precisely one
pair of complex conjugate points. For this we use the fact that, for
$C$ as in \ref{sosstable}, every psd element of $\R[C]$ is a sum of
squares in $\R[C]$ (see \cite{sch:mz}). The asserted reduction
follows therefore from the following lemma:

\begin{lem}\label{redlem}
Let $C$ be an affine curve over $\R$, and let $C'$ be a Zariski open
subset of $C$. Assume that every psd element of $\R[C]$ is a sum of
squares in $\R[C]$. If the preordering of sums of squares in $\R[C]$
is stable, then the preordering of sums of squares in $\R[C']$ is
stable as well.
\end{lem}

\begin{proof}
There exists $s\in\R[C]$ such that $\R[C']=\R[C]_s$, the ring of
fractions $\frac f{s^n}$ with $f\in\R[C]$ and $n\ge0$. Let $L'$ be a
finite-dimensional subspace of $\R[C']$, and choose $n\ge0$ such that
$s^{2n}L'=:L$ is contained in $\R[C]$. Since $\Sigma\R[C]^2$ is
stable in $\R[C]$, there is a finite-dimensional subspace $W$ of
$\R[C]$ such that every element of $L\cap\Sigma\R[C]^2$ is a sum of
squares of elements of $W$. Now let $f\in L'\cap\Sigma\R[C']^2$. Then
$s^{2n}f$ lies in $L$, and it is psd on $C(\R)$ since $C$ has no
isolated real points (the latter by \cite{sch:mz} Thm.\ 4.18).
By assumption, therefore, $s^{2n}f\in\Sigma\R[C]^2$, hence $s^{2n}f$
is a sum of squares of elements of $W$. So if we put $W':=s^{-n}W$,
every element of $L'\cap\Sigma\R[C']^2$ is a sum of squares of
elements of $W'$.
\end{proof}

\begin{lab}\label{suff}
So it suffices to consider a nonsingular affine curve $C$ of genus
one over $\R$ with precisely one pair $\infty$, $\ol\infty$ of
complex conjugate points at infinity. Note that this implies that
$C(\R)$ is compact. It follows from Riemann-Roch that $C$ is
isomorphic to a plane affine curve with equation
$$y^2+q(x)\>=\>0,$$
where $(x,y)$ are plane affine coordinates and $q(x)\in\R[x]$ is a
monic polynomial of degree~$4$ without multiple roots. We can also
assume that $q(x)$ is indefinite, i.e.\ has ($2$~or~$4$) real roots,
since otherwise $C(\R)$ is empty (in which case the theorem is both
true and uninteresting).
Note that $\ol C$, the nonsingular projective model of $C$, is the
normalization of the Zariski closure of $C$ in $\P^2$, and is an
elliptic curve over~$\R$.

Conversely, every plane affine curve over $\R$ with equation $y^2+
q(x)=0$ with $q$ monic and separable of degree four is nonsingular of
genus one and has precisely two complex conjugate points at infinity.
\end{lab}

\begin{lab}
From now on $C$ will always be a curve as in \ref{suff}. Usually we
shall not distinguish in our notation between a polynomial $f\in
\R[x,y]$ and its restriction to $C$ (i.e.\ the image under the
canonical map $\R[x,y]\to\R[C]$. Instead of working with the ordinary
(total) degree of polynomials we will use a variant which is better
adapted to the curve $C$:

Let $\R(C)$ be the (real) function field of $C$. Given any point
$p\in C(\C)$ we let $v_p\colon\R(C)^*\to\Z$ be the associated
discrete valuation of $\R(C)$. Given $f\in\R(C)$, we'll write
$$\delta(f)\>:=\>-v_\infty(f)=-v_{\ol\infty}(f)$$
(putting $\delta(0):=-\infty$). So $\delta$ is the negative of a
discrete valuation on $\R(C)$. For any $n\ge1$, the elements $x^i$
($0\le i\le n$) and $x^jy$ ($0\le j\le n-2$) form a linear basis of
the subspace $\{f\in\R[C]\colon\delta(f)\le n\}$ of $\R[C]$.
\end{lab}

In \cite{sch:tams} Sect.~4, it was proved that every psd element of
$\R[C]$ is a sum of squares in $\R[C]$. (At that time general results
like \cite{sch:mz} Thm.\ 4.18 were not yet available.) In order to
prove the stability result of Theorem \ref{sosstable}, we first need
to review a part of the proof from \cite{sch:tams} and analyze the
involved $\delta$-degrees. This is done in the next lemma:

\begin{lem}\label{keystep}
Let $0\ne f\in\R[C]$ be psd on $C(\R)$, and assume that $f$ has at
least one nonreal zero in $C(\C)$. Then there exist $g_1$, $g_2\in
\R[C]$ with
\begin{itemize}
\advance\parskip by2pt
\item[(a)]
$\delta(f-g_1^2-g_2^2)\le\delta(f)$;
\item[(b)]
$f-g_1^2-g_2^2$ has strictly less nonreal zeros than $f$ in $C(\C)$,
or is identically zero;
\item[(c)]
$\delta(g_1)$, $\delta(g_2)\>\le\>\bigl\lceil\frac12\delta(f)\bigr
\rceil$.
\end{itemize}
\end{lem}

Here the zeros of $0\ne f\in\R[C]$ in $C(\C)$ are counted with
multiplicities. As usual we write $\lceil x\rceil=\min\{n\in\Z\colon
x\le n\}$ for $x\in\R$.

\begin{proof}
Let $m=\delta(f)\ge1$. All divisors are calculated on the
complexified curve $\ol C_\C$, i.e.\ they are finite integral linear
combinations of the points in $\ol C(\C)$. We write
$$\div(f)=2D+\Theta-m(\infty+\ol\infty)$$
where $D$, $\Theta$ are conjugation-invariant effective divisors such
that the support of $D$ contains only real points and the support of
$\Theta$ contains no real point. By hypothesis $\Theta\ne0$. Let $p=
q=\frac m2$ if $m$ is even, and put $p=\frac{m+1}2$, $q=\frac{m-1}2$
if $m$ is odd. Then $p+q=m$, and the divisor $E:=-D+p\infty+
q\ol\infty$ satisfies $\deg(E)=\frac12\deg(\Theta)\ge1$.
By Riemann-Roch there exists $g\in\C(C)^*$ with $\div(g)+E\ge0$,
hence with
$$\div(g\ol g)\>\ge\>2D-m(\infty+\ol\infty).$$
It follows that $g\ol g\in\R[C]$, and the rational function $\varphi
:=g\ol g/f$ on $C$ has no poles in $C(\R)$. Let $c>0$ be the maximum
value that $\varphi$ takes on the compact set $C(\R)$, say $\varphi
(p)=c$ with $p\in C(\R)$. The regular function $h:=f-\frac1c\,g\ol g$
on $C$ is psd on $C(\R)$ and vanishes at $p$. From $\delta(g\ol g)\le
m$ we see $\delta(h)\le m$. Writing $\frac1{\sqrt c}g=g_1+ig_2$ with
$g_1$, $g_2\in\R[C]$ we have $\frac1c\,g\ol g=g_1^2+g_2^2$, and we
see
$$\delta(g_1),\ \delta(g_2)\>\le\>\max\{p,q\}\>=\>\Bigl\lceil\frac m2
\Bigr\rceil.$$
For every point $q\in C(\R)$ we have $v_q(h)\ge v_q(f)$, and even
$v_p(h)\ge2+v_p(f)$ if $q=p$.
Counting with multiplicity, $h$ has therefore strictly more real
zeros on $C$ than $f$. Since $\delta(h)\le\delta(f)$, we see that $h$
has strictly less nonreal zeros than $f$ (or else $h=0$).
\end{proof}

By applying Lemma \ref{keystep} inductively, we obtain the following
reduction to psd regular functions with only real zeros:

\begin{prop}\label{keycor}
Let $0\ne f\in\R[C]$ be psd. There are finitely many regular
functions $0\ne g_1,\dots,g_r\in\R[C]$ ($r\ge0$) such that
\begin{itemize}
\advance\parskip by2pt
\item[(a)]
$h:=f-(g_1^2+\cdots+g_r^2)$ is psd on $C(\R)$;
\item[(b)]
$h=0$, or all zeros of $h$ on $C$ are real;
\item[(c)]
$\delta(g_i)\le\bigl\lceil\frac12\delta(f)\bigr\rceil$ for $i=1,
\dots,r$, and $\delta(h)\le\delta(f)$.
\qed
\end{itemize}
\end{prop}

\begin{rem}
From Lemma \ref{keystep} we see that the number $r$ of squares in
Proposition \ref{keycor} can be bounded by the number of nonreal
zeros of $f$ in $C(\C)$, counted with multiplicities. In other words,
$r\le2(m-k)$ where $\delta(f)=m$ and $2k$ is the number of real zeros
of $f$, counted with multiplicities. On the other hand, it is well
known that the Pythagoras number of $\R[C]$ is $\le4$.
\end{rem}

The second step consists in studying the nonnegative regular
functions on $C$ with only real zeros. We will see that part of the
conclusions made in Example \ref{sex} for linear psd polynomials
generalizes to psd polynomials of any degree on $C$. Recall that $C$
has the affine equation $y^2+q(x)=0$ where the monic quartic
polynomial $q(x)\in\R[x]$ is square-free and indefinite. Let $\alpha<
\beta$ denote the smallest resp.\ the largest real zero of $q(x)$.
Let $\R(C)$ be the function field of $C$.

\begin{prop}\label{psdrzgp}
Let $G$ be the subgroup of $\R(C)^*/\R(C)^{*2}$ which is generated by
the cosets $f\R(C)^{*2}$ of all psd $0\ne f\in\R[C]$ which have only
real zeros on $C$. Then $G$ has order four and is generated by the
cosets of $x-\alpha$ and of $\beta-x$.
\end{prop}

\begin{proof}
Since the square classes of $x-\alpha$ and $\beta-x$ lie in $G$ and
are independent, it is enough to show $|G|=4$. This was done in
\cite{sch:tams} Prop.\ 4.3, where $|G|$ was calculated in a more
general setting.
\end{proof}

\begin{dfn}
Let $0\ne f\in\R[C]$. By $\theta(f)$ we denote the least integer
$d\ge0$ for which there exists a sums of squares representation $f=
f_1^2+\cdots+f_r^2$ with $r\in\N$ and $f_i\in\R[C]$ such that $\delta
(f_i)\le d$ for $i=1,\dots,r$. We put $\theta(f)=\infty$ if $f$ is
not a sum of squares in $\R[C]$.
\end{dfn}

Note that one obviously has $\theta(f+g)\le\max\{\theta(f),\,\theta
(g)\}$ and $\theta(fg)\le\theta(f)+\theta(g)$.

\begin{lem}\label{lemred2gens}
Let $0\ne f$, $g\in\R[C]$ be psd. Assume that $g$ has only real zeros
on $C$ and that $f/g$ is a square in $\R(C)$. If $g=b_1^2+\cdots+
b_r^2$ with $b_1,\dots,b_r\in\R[C]$,
then there exist $a_1,\dots,a_r\in\R[C]$ with $f=a_1^2+\cdots+a_r^2$
and with
$$\delta(a_i)\>=\>\delta(b_i)+\frac12\Bigl(\delta(f)-\delta(g)
\Bigr)$$
($i=1,\dots,r$). In particular we have
$$2\theta(f)-\delta(f)\>\le\>2\theta(g)-\delta(g).$$
\end{lem}

\begin{proof}
Let $h\in\R(C)^*$ with $\frac fg=h^2$. We have $f=\sum_i(b_ih)^2$, so
it suffices to show that $a_i:=b_ih$ lies in $\R[C]$ and $\delta
(a_i)$ satisfes the identity of the lemma ($i=1,\dots,r$). Every pole
of $a_i$ on $C$ is a zero of $g$, so it is real by the assumption. On
the other hand, for $i=1,\dots,r$ and for every point $p\in C(\R)$ we
have $v_p(g)\le2v_p(b_i)$, hence $v_p(h)\ge\frac12v_p(f)-v_p(b_i)$
and $v_p(a_i)\ge\frac12v_p(f)\ge0$. This proves $a_i\in\R[C]$.
Clearly $\delta(a_i)=\delta(b_ih)=\delta(b_i)+\frac12(\delta(f)
-\delta(g))$, and this implies $\theta(f)\le\theta(g)+\frac12(\delta
(f)-\delta(g))$.
\end{proof}

In Lemma \ref{lemred2gens}, note that we have in fact $\theta(f)
-\theta(g)=\frac12(\delta(f)-\delta(g))$ if both $f$ and $g$ have only
real zeros.

This discussion leads to the following result. It completes the proof
of Theorem \ref{sosstable}:

\begin{thm}\label{thmsum}
Let $q$ be a quartic monic polynomial which is indefinite and has no
multiple roots, and let $C$ be the affine curve $y^2+q(x)=0$ over
$\R$. There is an integer $N\ge1$ such that
$$\theta(f)\>\le\>N+\Bigl\lceil\frac12\delta(f)\Bigr\rceil$$
holds for every psd regular function $f$ in $\R[C]$.
\end{thm}

\begin{proof}
Let $\alpha$ (resp.\ $\beta$) be the smallest (resp.\ largest) real
zero of $q(x)$, and put $l_1=x-\alpha$, $l_2=\beta-x$. Each of $l_1$,
$l_2$ and $l_1l_2$ has only real zeros on $C$ and is a sum of squares
in $\R[C]$ (\cite{sch:tams} Thm.\ 4.10(a)). We claim that the theorem
holds with $N=\max\{\theta(l_1),\,\theta(l_2),\,\theta(l_1l_2)\}$.

To see this let $0\ne f\in\R[C]$ be psd. By Prop.\ \ref{keycor} there
exists a psd element $h\in\R[C]$ which is either identically zero or
has only real zeros on $C$, such that $\theta(f-h)\le\lceil\frac12
\delta(f)\rceil$ and $\delta(h)\le\delta(f)$. We can assume $h\ne0$.
By Proposition \ref{psdrzgp} there is $g\in\{1,\,l_1,\,l_2,\,l_1
l_2\}$ such that $h/g$ is a square in $\R(C)^*$, and by Lemma
\ref{lemred2gens} we have $\theta(h)\le\theta(g)+\frac12(\delta(h)
-\delta(g))$.
So we conclude $\theta(h)\le N+\frac12(\delta(f)+1)$. Hence the same
bound holds for $\theta(f)$ since $\theta(f)\le\max\{\theta(h)$,
$\theta(f-h)\}$.
\end{proof}

\begin{rem}\label{remthmsum0}
Theorem \ref{thmsum} is a sharpening of the stability assertion of
Theorem \ref{sosstable}, as far as the plane curves $y^2+q(x)=0$ are
concerned that are considered in \ref{thmsum}. We would like to point
out that \ref{thmsum} also yields a similar sharpening for the other
curves discussed in \ref{sosstable}. Indeed, the reduction lemma
\ref{redlem} and its proof are explicit enough to permit a transfer
of the assertion of \ref{thmsum} to Zariski open subcurves. Although
we won't make this more explicit, it justifies to restrict the
remaining discussions to plane curves as in \ref{thmsum}.
\end{rem}

\begin{rem}\label{remthmsum}
A closer inspection of the last proof exhibits that Theorem
\ref{thmsum} is true with
$$N\>=\>\theta\bigl((x-\alpha)(\beta-x)\bigr)-1,$$
where $\alpha$ is the smallest and $\beta$ is the largest real root
of $q(x)$. Clearly, this is the smallest possible $N$, as we see by
taking $f=(x-\alpha)(\beta-x)$ in Theorem \ref{thmsum}.

Indeed, let us abbreviate $l_1=x-\alpha$ and $l_2=\beta-x$. From $l_1
+l_2=\beta-\alpha$ we get $(\beta-\alpha)l_1=l_1^2+l_1l_2$, which
implies $\theta(l_1)\le\theta(l_1l_2)$. Similarly $\theta(l_2)\le
\theta(l_1l_2)$. Let us distinguish the argument according to the
parity of $\delta(h)$. If $\delta(h)$ is even then $g=1$ or $g=l_1
l_2$, and $g=1$ gives the bound $\theta(h)\le\frac12\delta(h)\le
\frac12\delta(f)$, while $g=l_1l_2$ gives the bound $\theta(h)\le
\theta(l_1l_2)+\frac12\delta(h)-1\le\theta(l_1l_2)+\frac12\delta(f)
-1$. If $\delta(h)$ is odd then $g=l_j$ for $j\in\{1,2\}$, and this
gives the bound $\theta(h)\le\theta(l_j)+\frac12(\delta(h)-1)$,
which is at most $\theta(l_1l_2)+\lceil\frac12\delta(f)\rceil-1$.
\end{rem}

\begin{dfn}\label{dfnnc}
Let $C$ have equation $y^2+q(x)=0$ with $q$ monic, separable and
indefinite of degree four, and let $\alpha<\beta$ be the smallest
resp.\ largest real root of $q$. We write
$$N_C\>:=\>\theta\bigl((x-\alpha)(\beta-x)\bigr)\ \in\N,$$
and we call $N_C$ the \emph{stability constant} of the curve $C$.
\end{dfn}

Remark \ref{remthmsum} has shown:

\begin{cor}\label{corthmsum}
For every psd $f\in\R[C]$ we have
$$\theta(f)\>\le\>N_C-1+\Bigl\lceil\frac12\delta(f)\Bigr\rceil.
\eqno\square$$
\end{cor}

\begin{cor}
In the terminology of \cite{gpt}, the ideal $\scrI_C=(y^2+q(x))$ of
$C$ in $\R[x,y]$ is $(d,\,N_C+d-1)$-sos for every $d\ge1$. In
particular, this ideal is theta-exact of theta-rank $N_C$.
\end{cor}

\begin{proof}
Let $p\in\R[x,y]$ have degree $d$, let $\ol p=p+\scrI_C\in\R[C]$. We
have $\delta(\ol p)\le2d$, so if $p$ is psd on $C(\R)$, Corollary
\ref{corthmsum} shows that $p\equiv\sum_jp_j(x,y)^2$ (mod~$\scrI_C$)
where $\delta(\ol p_j)\le N_C+d-1$ for every $j$. Thus every $p_j$ is
congruent modulo $\scrI_C$ to a polynomial of degree $\le N_C+d-1$,
which proves the corollary.
\end{proof}

In the next section we shall study in more detail how $N_C$ depends
on the curve $C$, i.e.\ on the polynomial $q(x)$). In particular, we
will see that $N_C$ can become arbitrarily large.

\begin{rem}\label{sex}
Assume $f\in\R[x,y]$ is a \emph{linear} polynomial that is
nonnegative on $C(\R)$, where $C$: $y^2+q(x)=0$ is a curve as in
Theorem \ref{thmsum}. In this case we can make the argument leading
to the proof of the theorem entirely explicit. We assume that $f$ has
a real zero $p=(\xi,\eta)$ in $C(\R)$. So $f=0$ is the tangent line
to the plane curve $C$ at the point $p$.

Let us first assume that $\eta\ne0$ (the tangent is not vertical),
and that $f=0$ is not a double tangent. Then $f$ has a pair of
complex conjugate nonreal zeros on $C$, and we can apply the
construction from Lemma \ref{keystep} with $g=x-\xi$. The rational
function
$$\varphi(x,y)\>=\>\frac{(x-\xi)^2}{f(x,y)}$$
has no poles on $C(\R)$; let $\gamma>0$ be its maximum value. Then
$h:=f-\frac1\gamma(x-\xi)^2$ is psd on $C(\R)$ and has only real
zeros on $C$. If $q\in C(\R)$ is the point where $\varphi$ attains
its maximum $\gamma$, then the conic $h(x,y)=0$ is tangent to $C$ in
the points $p$ and $q$.
The psd function $h$ lies in the square class of $(x-\alpha)
(\beta-x)$ in $\R(C)^*/\R(C)^{*2}$. More explicitly, we have
\begin{equation}\label{explicsqcl}
h\>=\>\text{const}\cdot\frac{F^2}{(x-\alpha)(\beta-x)}
\end{equation}
with a positive constant and with
\begin{equation}\label{explicsqel2}
F\>=\>(\xi^2y-\eta x^2)+(\alpha+\beta)(\eta x-\xi y)+\alpha\beta
(y-\eta).
\end{equation}
Indeed, the above $F$ is nonzero since $\eta\ne0$ and has $\delta(F)
\le2$, and $F$ vanishes in $(\alpha,0)$, $(\beta,0)$ and $p=(\xi,
\eta)$. If we call $\tilde q$ the fourth zero of $F$, then the
rational function on the right of \eqref{explicsqcl} has zero divisor
$2(p+\tilde q)$, while $h$ has zero divisor $2(p+q)$. This implies
$q=\tilde q$ unless $q$ and $\tilde q$ are $(\alpha,0)$ and $(\beta,
0)$, which is excluded by the assumption $\eta\ne0$. Note that
$\tilde q=q$ is the point where $\varphi$ attains its maximum.

If $f=0$ is a double tangent then $\div(f)=2(p+q-\infty-\ol\infty)$
with a real point $q$ on $C$ (possibly $q=p$), and the argument of
the first case remains formally true (with $\gamma=\infty$, i.e.\
with $h=f$). So in this case
$$f\>=\>\text{const}\cdot\frac{F^2}{(x-\alpha)(\beta-x)}$$
with a positive constant and with $F$ as in \eqref{explicsqel2}.

In summary, once we have an explicit representation
$$(x-\alpha)(\beta-x)\>=\>\sum_\nu g_\nu^2$$
as a sum of squares in $\R[C]$, we immediately get an \emph{explicit}
sum of squares representation for every psd tangent line $f$ to $C$,
namely
$$f\>=\>\frac1\gamma(x-\xi)^2+{\rm const}\cdot\sum_\nu
\Bigl(\frac{Fg_\nu}{(x-\alpha)(\beta-x)}\Bigr)^2$$
with $F$ as in \eqref{explicsqel2}. (This is correct when $\eta\ne0$
and $f=0$ is not a double tangent; when $f=0$ is a double tangent it
is true with $\gamma=\infty$; when $\eta=0$ it is true with $F=
(x-\alpha)(\beta-x)$.) All fractions on the right lie in $\R[C]$.
\end{rem}

\begin{rems}
\hfil
\smallskip

1.\
It is not known whether Theorem \ref{sosstable} extends to curves
of genus greater than one. For simplicity, let us restrict the
discussion to irreducible affine and nonsingular curves $C$ over
$\R$ with $C(\R)\ne\emptyset$. When all points of $C$ at infinity
are real, then the sums of squares (sos) cone in $\R[C]$ is known
to be stable \cite{ps}. However, as soon as the genus $g_C\ge1$, this
assumption implies that the sos cone in $\R[C]$ is (much) smaller
than the psd cone \cite{sch:tams}. On the other side, when $C$ has
nonreal points at infinity (for example, when $C(\R)$ is compact),
then the psd and the sos cone in $\R[C]$ coincide \cite{sch:mz}.
However, there is not a single such curve of genus $\ge2$ for which
it is known whether or not the sos cone is stable.
\smallskip

2.\
It is natural to weaken the question, and to ask only for partial
stability, as in \ref{condsex}(2). For example, when $C$ is a plane
nonsingular curve of genus greater than one with $C(\R)$ compact, can
every linear polynomial nonnegative on $C(\R)$ be written as a sum of
squares in $\R[C]$, with the degrees of the summands bounded
uniformly? Of course, this would be much weaker a property than full
stability, and perhaps the answer is not so hard.
\end{rems}

%-------------------------------------------------------------------%

\section{Application: Lifted LMI representations}

Here we sketch how the main results of the previous section, combined
possibly with further explicit results on degree bounds from the next,
lead to very explicit lifted LMI representations of the convex hull
of the curves considered.

First we record:

\begin{cor}
Let $C\subset\A^n$ be an irreducible real curve of genus one for
which $C(\R)$ is compact. Then the convex hull of $C(\R)$ in $\R^n$
has a lifted LMI representation.
\end{cor}

\begin{proof}
Let $\tilde C\to C$ be the normalization of $C$. Since Theorem
\ref{sosstable} applies to $\tilde C$, the Lasserre relaxation
construction \ref{lasrelax} becomes exact on every finite-dimensional
linear subspace $L$ of $\R[\tilde C]$. Let $\R[\x]_1=\{f\in\R[\x]
\colon\deg(f)\le1\}$, and perform the relaxation construction to the
image $L$ of $\R[\x]_1$ under $\R[\x]\onto\R[C]\into\R[\tilde C]$.
\end{proof}

We wish to demonstrate the explicitness of the construction by two
examples. For this we restrict to discussing plane affine curves with
equation $y^2+q(x)=0$ as in \ref{thmsum}.

\begin{rem}\label{explsdp}
Let $C$ be the curve $y^2+q(x)=0$, and let $L=\R[C]_1$ be the
subspace of $\R[C]$ spanned by $1$, $x$ and $y$. Let $N:=N_C$ be the
stability constant of $C$ (\ref{dfnnc}). By Corollary
\ref{corthmsum}, Lasserre's relaxation construction \ref{lasrelax}
works using the subspaces $W=\{f\colon\theta(f)\le N\}$ and $U=
\{f\colon\theta(f)\le2N\}$ of $\R[C]$. Since $\dim(W)=2N$ and $\dim
(U)=4N$, this presents the convex hull of $C(\R)$ in $\R^2$ in the
form
$$\conv C(\R)\>=\>\Biggl\{(x,y)\in\R^2\colon\ex z_1,\dots,z_{2N-3}
\text{ with }xA+yB+C_0+\sum_{j=1}^{4N-3}z_jC_j\succeq0\Biggr\}$$
where $A$, $B$, $C_j$ ($j=0,\dots,4N-3$) are real symmetric matrices
of size $2N\times2N$ that are easy to make explicit. (Here $S\succeq
0$ means that the symmetric matrix $S$ is positive semidefinite.)
\end{rem}

\begin{example}
For an illustration, consider the simplest case, which is curves $C$
with $N_C=2$. Up to a linear coordinate change, these are precisely
the curves with equation $y^2+(x^2-1)(x^2+b)=0$ where $b\ge-1$, $b\ne
0$ (see \ref{N23} below). If this equation is written $y^2+x^4+Ax^2+B
=0$, then $\conv C(\R)$ is the set of $(x,y)\in\R^2$ for which there
are $u_1,u_2,u_3,v_1,v_2\in\R$ such that
$$\left(\begin{array}{cccc}1&x&u_2&y\\x&u_2&u_3&v_1\\u_2&u_3&u_4&v_2
\\y&v_1&v_2&-B-Au_2-u_4\end{array}\right)\ \succeq\ 0.$$
This matrix is obtained using the basis $1,x,x^2,y$ of $W$ and the
basis $x^j$, $x^ky$ ($0\le j\le4$, $0\le k\le2$) of $U$, resp.\ its
dual basis of $U^\du$.
\end{example}

\begin{example}
As pointed out in \ref{verorem}, we can expect interesting results as
well from using construction \ref{lasrelax} for subspaces $L$
different from $\R[C]_1$. For example, we get concrete descriptions
of the convex hulls of embeddings of $C$ into higher-dimensional
spaces, or of singular quotients of the curve, or of combinations of
both. To present one more illustration, consider the curve $y^2+x^4
=1$, and perform construction \ref{lasrelax} with the subspace $L$ of
$\R[C]$ spanned by $1$, $x$ and $xy$. This gives the ``figur eight''
curve
$$C':\quad w^2=x^2(1-x^4)$$
and its convex hull in the $(x,w)$-plane. Since $N_C=2$ (see
\ref{N23} below), every psd element $f$ of $L$ satisfies $\theta(f)
\le3$ by \ref{corthmsum}, so the construction works with $W=\{f\colon
\delta(f)\le3\}$ and $U=\{f\colon\delta(f)\le6\}$. This yields a
lifted LMI representation of $\conv C'(\R)$ by symmetric $6\times6$
matrices with $9$~free variables, namely as the set of $(x,w)\in\R^2$
for which there exist real numbers $u_j$ ($2\le j\le6$) and $v_j$
($j\in\{0,2,3,4\}$) making the matrix
$$\left(\begin{array}{cccccc}
1 & x & u_2 & u_3 & v_0 & w \\
x & u_2 & u_3 & u_4 & w & v_2 \\
u_2 & u_3 & u_4 & u_5 & v_2 & v_3 \\
u_3 & u_4 & u_5 & u_6 & v_3 & v_4 \\
v_0 & w & v_2 & v_3 & 1-u_4 & x-u_5 \\
w & v_2 & v_3 & v_4 & x-u_5 & u_2-u_6
\end{array}\right)$$
nonnegative.
\end{example}

\begin{rem}
As far as we are aware, this is the first example in the literature
where explicit semidefinite representations are given for convex
hulls of non-rational real algebraic varieties. For rational curves,
such representations were given by Parrilo (\cite{paba}, unpublished)
and by Henrion \cite{he}, who also treats the quadratic Veronese
surface. The arguments in these cases are elementary.
\end{rem}

%-------------------------------------------------------------------%

\section{Degree bounds: A detailed study}

Since explicit bounds for the stability constant (see \ref{dfnnc})
are necessary for produce concrete lifted LMI presentations, see
Remark \ref{explsdp}, we think it worthwile to discuss this constant
and its dependence on the individual curve in greater detail.

\begin{lab}
We keep the assumptions of \ref{suff}. So $q\in\R[x]$ is a monic
quartic polynomial which is indefinite and separable, and $C$ is the
affine real curve with equation $y^2+q(x)=0$. Let $\alpha<\beta$ be
the smallest resp.\ the largest real root of $q$, write $f=(x-\alpha)
(x-\beta)$, and let $h\in\R[x]$ be the monic quadratic polynomial
with $q=fh$. We have seen that the stability constant $N_C=\theta
(-f)$ governs all degree bounds for sums of squares decompositions
in $\R[C]$ (\ref{remthmsum}).
\end{lab}

\begin{lem}\label{umschreib}
Let $d$ be the smallest number for which there is an identity
$$th-sf\>=\>1$$
with sums of squares $s$, $t$ in $\R[x]$ and with $\deg(s)=\deg(t)\le
d$. Then $N_C=\frac d2+2$.
\end{lem}

\begin{proof}
Since $-f=\sum_i(a_i+b_iy)^2$ with $a_i$, $b_i\in\R[x]$ implies $-f=
\sum_ia_i^2-q\sum_ib_i^2$, we only need to consider identities
$$-f\>=\>s'-tq$$
with $a_i$, $b_i\in\R[x]$ and $s'=\sum_ia_i^2$, $t=\sum_ib_i^2$.
Clearly,
\begin{equation}\label{thetabsch}
\max_i\bigl\{\theta(a_i),\>\theta(b_iy)\bigr\}\>=\>\frac12\deg(s')\>=
\>2+\frac12\deg(t).
\end{equation}
Since $f$ has only real zeros, $f$ necessarily divides every $a_i$,
and so $f^2$ divides $s'$. Dividing by $f$ and putting $s=s'/f^2$ (a
sum of squares in $\R[x]$) we get $-1=sf-th$. The lemma follows.
\end{proof}

\begin{lab}\label{normalizeq}
By a linear change of variables we can normalize the equation of $C$
so that it becomes
\begin{equation}\label{cab}
y^2+(x^2-1)h(x)\>=\>0,\quad h(x)=x^2+ax+b,
\end{equation}
where $h$ is separable and $h(x)>0$ for $|x|\ge1$.
So the smallest (resp.\ the largest) real root of $q(x)=(x^2-1)h(x)$
are $-1$ (resp.~$+1$). For our study of how the stability constant
$N_C$ depends on the curve $C$, it will be convenient to assume that
$C$ has this normalized form. The conditions on $h$ mean that $(a,b)$
lies in the set
$$P\>:=\>\bigl\{(a,b)\in\R^2\colon a^2-4b<0\ \lor\ (a^2-4b>0\ \land\
|a|<\min\{2,\,|b|+1\})\bigr\}.$$
Let us denote by $C_{a,b}$ the affine curve with equation
\eqref{cab}, and let us abbreviate its stability constant by
$$N(a,b)\>:=\>N_{C_{a,b}},$$
for $(a,b)\in P$.
\end{lab}

\begin{lab}\label{N23}
By Lemma \ref{umschreib} we have $N(a,b)=2+\frac12\deg(t)$, where
$s$, $t\in\R[x]$ are psd polynomials with
$$1\>=\>t(x)\,(x^2+ax+b)-s(x)\,(x^2-1)$$
and $\deg(s)=\deg(t)$ is as small as possible. It is therefore clear
that always $N(a,b)\ge2$ holds, and that $N(a,b)=2$ if and only if
$a=0$. Without proof we remark that $N(a,b)\le3$ if and only if
$$\frac{a^4}{16}+a^2\>\le\>(b+1)^2.$$
The following picture shows the parameter set $P$, the yellow part
corresponding to $N\le3$ and the red part to $N\ge4$:
\begin{center}
\includegraphics[height=60mm]{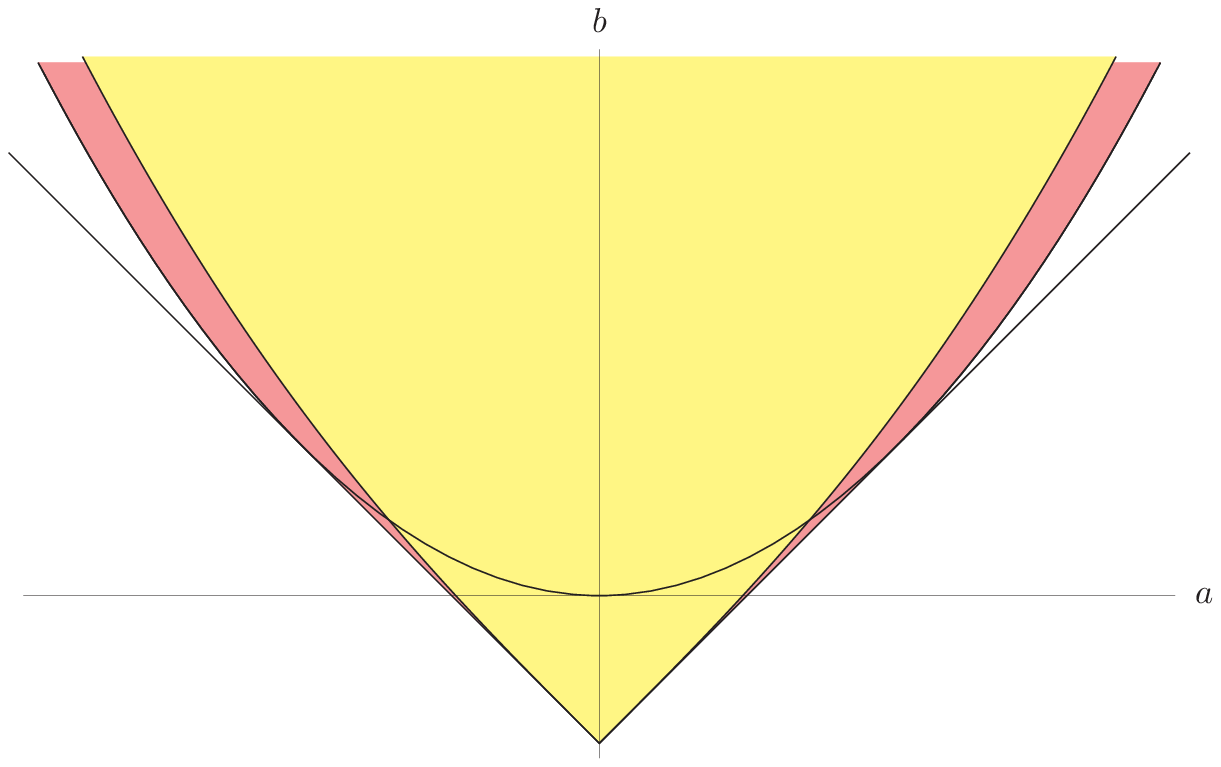}
\end{center}
\end{lab}

For the next lemma let $P'$ denote the boundary of the closure of
$P$, so
$$P'\>=\>\bigl\{(a,b)\in\R^2\colon a^2=4b\ge4\ \lor\ |a|=b+1\le2
\bigr\}.$$

\begin{lem}
Let $(a_\nu,b_\nu)_{\nu\ge1}$ be a sequence in $P$ that converges to
$(a,b)\in\R^2$ for $\nu\to\infty$. If the sequence $N(a_\nu,b_\nu)$
is bounded, and if $(a,b)\ne(0,-1)$, then $(a,b)\notin P'$. If in
addition $(a,b)\in P$ then $N(a,b)\le\sup_\nu N(a_\nu,b_\nu)$.
\end{lem}

(If $(a,b)\ne P\cup P'$ then $a^2-4b=0$ and $|a|<2$.)

\begin{proof}
Assume that the sequence $N(a_\nu,b_\nu)$ is bounded. By \ref{N23}
this means that there are $d\ge0$ and sums of squares $s_\nu(x)$,
$t_\nu(x)$ in $\R[x]$ with $\deg(s_\nu)=\deg(t_\nu)\le2d$ and
\begin{equation}\label{convergeq}
1\>=\>(x^2+a_\nu x+b_\nu)t_\nu(x)-(x^2-1)s_\nu(x)
\end{equation}
for every $\nu$. We first assume that the coefficients of the $t_\nu$
and $s_\nu$ are uniformly bounded for all $\nu$. After passing to a
suitable subsequence we can then assume that we have
(coefficient-wise) convergences $s_\nu\to s$ and $t_\nu\to t$, where
$s$, $t\in\R[x]$ are clearly sums of squares. Passing
\eqref{convergeq} to the limit $\nu\to\infty$ we see
\begin{equation}\label{limeq}
1\>=\>(x^2+ax+b)t(x)-(x^2-1)s(x).
\end{equation}
If $(a,b)\in P$, it follows that $N(a,b)\le d$. Assume $(a,b)\in P'$
and $(a,b)\ne(0,-1)$. If $|a|>2$ then $a^2=4b$, so $x^2+ax+b=(x+\frac
a2)^2$ has a double zero at $-\frac a2$, which contradicts
\eqref{limeq}.
If $|a|\le2$ then $|a|=b+1$, and so $x^2+ax+b$ has a zero at $\pm1$,
which equally contradicts \eqref{limeq}.

There remains the case where the coefficients or $s_\nu$ or $t_\nu$
are unbounded for $\nu\to\infty$. We scale \eqref{convergeq} for each
$\nu$ by the factor $\frac1{c_\nu}$ where $c_\nu>0$ is the maximum
absolut value of the coefficients of $s_\nu(x)$ and $t_\nu(x)$. After
passing to a subsequence we have convergence $c_\nu^{-1}s_\nu(x)\to
s(x)$ and $c_\nu^{-1}t_\nu(x)\to t(x)$, and again $s$, $t\in\R[x]$
are sums of squares. Both are nonzero since each has a coefficient
$\pm1$. Taking \eqref{convergeq} to the limit gives
$$(x^2+ax+b)\,t(x)\>=\>(x^2-1)\,s(x).$$
This implies that $(x^2-1)(x^2+ax+b)$ is a psd polynomial, which only
happens for $(a,b)=(0,-1)$.
\end{proof}

\begin{cor}\label{boundexplode}
\hfil
\begin{itemize}
\item[(a)]
For each $N\ge0$, the set $\{(a,b)\in P\colon N(a,b)\le N\}$ is
relatively closed in $P$.
\item[(b)]
When $(a,b)$ moves in $P$ towards a boundary point $\ne(0,-1)$ in
$P'$, then $N(a,b)$ tends to infinity.
\qed
\end{itemize}
\end{cor}

Note that (b) is not necessarily true when $(a,b)\to(0,-1)$, for
example since $N(0,b)=2$ for all $b>-1$.

\begin{rem}
Degeneration of $(a,b)\in P$ towards a boundary point $(a_0,b_0)\in
P'$, $(a_0,b_0)\ne(0,-1)$, corresponds to degenerating the curve
$C_{a,b}$ into a nodal curve (for $|a_0|\ne1$) or a cuspidal curve
(for $a_0=\pm1$), rational in either case.
\end{rem}

\begin{lab}\label{markov}
We do not know how to express $N(a,b)$ for arbitrary $(a,b)\in P$.
We conclude with proving an explicit lower bound for $N(a,b)$ in the
$|a|>2$ part. We keep the normalizations \ref{normalizeq} and write
$h=x^2+ax+b$ and $f=x^2-1$.

Assume that one of $h'(-1)>0$ or $h'(1)<0$ holds. Either condition
implies $h(x)>0$ for all $x\in\R$.
Let us assume $h'(1)<0$, and let $s$, $t\in\R[x]$ be psd polynomials
with $1=th-sf$ (c.f.\ \ref{umschreib}). We conclude
\begin{equation}\label{vgl}
t(x)\ge\frac1{h(x)}\text{ \ for }|x|\ge1,\quad0\le t(x)\le\frac
1{h(x)}\text{ \ for }|x|\le1.
\end{equation}
In particular $t(1)=\frac1{h(1)}$.
Since $h$ is quadratic, $h'(1)<0$ implies $h'(x)<0$ for all $x\le1$,
and so $\frac1h$ is strictly increasing for $x\le1$. Hence $0\le t(x)
\le t(1)=\frac1{h(1)}$ for $|x|\le1$. On the other hand, \eqref{vgl}
implies $t'(1)\ge(\frac1h)'(1)=-\frac{h'(1)}{h(1)^2}$.

According to Markov's inequality (\cite{mv}, \cite{be}), any
polynomial $p\in\R[x]$ of degree $\le n$ satisfies
$$||p'||_{[-1,1]}\>\le\>n^2\cdot||p||_{[-1,1]},$$
where $||p||_{[-1,1]}=\max\{|p(x)|\colon|x|\le1\}$. Applying this to
$p=t-\frac1{2h(1)}$ we conclude
\begin{equation}\label{markovest}
\deg(t)^2\>\ge\>2\cdot\frac{(1/h)'(1)}{(1/h)(1)}\>=\>-2\,\frac{h'(1)}
{h(1)}.
\end{equation}
Writing $h=x^2+ax+b$, the assumption $h'(1)<0$ means $a+2<0$, and
\eqref{markovest} becomes
$$\deg(t)^2\>\ge-\frac{2(a+2)}{1+a+b}.$$
If instead of $h'(1)<0$ we assume $h'(-1)>0$, we get a symmetric
estimate. Altogether we have shown:
\end{lab}

\begin{prop}\label{markovbd}
Consider the affine curve $y^2+(x^2-1)(x^2+ax+b)=0$ with $(a,b)\in
P$. If $|a|>2$ then
$$N(a,b)\>\ge\>2+\sqrt{\frac{|a|-2}{2(1+b-|a|)}}\>.\eqno\square$$
\end{prop}

\begin{example}
For $\gamma>0$ consider the curve $C_\gamma$ with equation $y^2+
(x^2-1)h_\gamma(x)=0$ where
$$h_\gamma(x)\>=\>x^2+\Bigl(2+\frac2\gamma\Bigr)x+\Bigl(1+\frac
2\gamma+\frac4{\gamma^2}\Bigr)\>=\>\Bigl(x+1+\frac1\gamma\Bigr)^2+
\frac3{\gamma^2}.$$
Via Markov's inequality we get the lower bound
$$N_{C_\gamma}\>\ge\>2+\frac{\sqrt\gamma}2$$
from Proposition \ref{markovbd}, which tends to infinity for $\gamma
\to\infty$.

However, this bound does not seem to come close to being sharp. We
did a small series of numerical experiments using Parrilo's
\texttt{sostools} package \cite{psosto}, resulting in the following
observations:
$$\begin{array}{c|c|r}
N & 4(N-2)^2 & \gamma_{\max}(N) \\
\hline
3 & 4 & 2.57 \\
4 & 16 & 6.92 \\
5 & 36 & 12.95 \\
6 & 64 & 20.70 \\
7 & 100 & 30.17 \\
8 & 144 & 41.35 \\
9 & 196 &\ \ 54.25
\end{array}$$
For given $N\in\N$ let $\gamma_{\max}(N)$ be the maximal $\gamma>0$
for which $N_{C_\gamma}\le N$. The Markov estimate gives $\gamma_
{\max}(N)\le4(N-2)^2$, which is the second column. The approximate
true value of $\gamma_{\max}(N)$ is shown in the last column.
\end{example}

%===================================================================%

\end{document}